\documentclass[twoside, preprint]{imsart}
\usepackage[sort, round]{natbib}
\usepackage[english]{babel}
\usepackage[a4paper]{geometry}
\usepackage{amsmath,amsfonts,amsthm, amssymb}
\usepackage{graphicx}
\usepackage{subfigure}
\usepackage{enumitem}

\RequirePackage[colorlinks,citecolor=blue,urlcolor=blue]{hyperref}

\startlocaldefs
\newcommand{\E}{\mathbb{E}}
\newcommand{\1}{\textbf{1}}
\newcommand{\var}{{\rm Var}}
\newcommand{\Oab}{O_{ab}(e)}

\newcommand{\nab}{n_{ab}(e)}

\newcommand{\LL}{\mathbb{L}}

\newcommand{\wt}{\widetilde}
\newcommand{\Diag}{{\rm Diag}}

\renewcommand{\d}{\delta}
\newcommand{\argmax}{\mathop{\rm argmax}\limits}
\def\eps{\epsilon}

\newcommand{\floor}[1]{\lfloor{#1}\rfloor}

\def\diag{\mathop{\rm diag\,}\nolimits}
\def\ra{\rightarrow}
\def\var{\mathop{\rm var}\nolimits}
\def\l{{\lambda}}

\theoremstyle{plain}
\newtheorem{thm}{Theorem}[section]
\newtheorem{lem}[thm]{Lemma}

\graphicspath{{./Afbeeldingen/}}

\endlocaldefs

\begin{document}

\begin{frontmatter}

\title{Bayesian Community Detection} 
\runtitle{Bayesian Community Detection}

\begin{aug}
  \author{\fnms{S.L.}  \snm{van der Pas}\corref{}\thanksref{t1}\ead[label=e1]{svdpas@math.leidenuniv.nl}}
  \and
  \author{\fnms{A.W.} \snm{van der Vaart}\thanksref{t2}\ead[label=e2]{avdvaart@math.leidenuniv.nl}}

  \thankstext{t1}{Research supported by Netherlands Organization for Scientific
Research NWO.}
\thankstext{t2}{The research leading to these results has received funding from the
European
  Research Council under ERC Grant Agreement 320637.}

  \runauthor{S.L. van der Pas and A.W. van der Vaart}

  \address{Mathematical Institute, Leiden University, \\ 
          \printead{e1,e2}}

\end{aug}

\begin{abstract}
We introduce a Bayesian estimator of the underlying class structure in the stochastic block model, when the number of classes is known. The estimator is the posterior mode corresponding to a Dirichlet prior on the class proportions, a generalized Bernoulli prior on the class labels, and a beta prior on the edge probabilities. We show that this estimator is strongly consistent when the expected degree is at least of order $\log^2{n}$, where $n$ is the number of nodes in the network. 
\end{abstract}



\end{frontmatter}

\section{Introduction}
The stochastic block model (SBM) \citep{Holland1983}  is a model for network data in which individual nodes are considered members of classes or communities, and the probability of a connection occurring between two individuals depends solely on their class membership. It has been applied to social, biological and communication networks, for example in \cite{Park2012}, \cite{Bickel2009} and \cite{Snijders1997} amongst many others. There are many extensions of the SBM for various applications, including the degree-corrected SBM \citep{Karrer2011, Zhao2012} which accounts for possible heterogeneity among nodes within the same class, and the mixed-membership SBM \citep{Airoldi2008}, in which  the assumption that the classes are disjoint is removed. These extensions allow for additional modelling flexibility. 

Two main SBM research directions are the recovery of the class labels (\emph{community detection}) and recovery of the remaining model parameters, consisting of the probability vector generating the class labels, and the class-dependent probabilities of creating an edge between nodes. In this paper, we focus on community detection, noting that once strong consistency of a community detection method has been established, consistency of the natural plug-in estimators for the remaining parameters follows directly by results in \citep{Channarond2012}. 

A large number of methods for recovering the class labels has been proposed. Those most closely related to this work are the modularities. \cite{Newman2004} introduced the term \emph{modularity} for `a measure of the quality of a particular division of a network'. They described one such measure for models in which edges are more likely to occur within classes than between classes, in which case there is a community structure in the colloquial sense, although the SBM does not require this assumption. \cite{Bickel2009} studied more general modularities, defining them as functions of the number of connections between all combinations of classes and the proportion of nodes placed in each class. They introduced the likelihood modularity, and provided general conditions under which modularities are consistent. Their method and theory was extended to the degree-corrected SBM by \cite{Zhao2012}. 

Spectral methods for community detection have gained in popularity, and refined results on error bounds are now available for the SBM and extensions of the SBM, as evidenced in \cite{Rohe2011}, \cite{Jin2015}, \cite{Sarkar2015} and \cite{Lei2015} for example. Many other algorithms have been introduced, most of them currently lacking formal proofs of consistency. A notable exception is the Largest Gaps algorithm \citep{Channarond2012}, which only takes the degree of each node as its input, and is strongly consistent under a separability condition.

A Bayesian approach towards recovering the class assignments in the SBM was first suggested by \cite{Snijders1997}, motivated by computational advantages of Gibbs sampling  over maximum likelihood estimation. They considered two classes and proposed  uniform priors on the class proportions and the edge probabilities. This approach was extended in \citep{Nowicki2001} to allow for more classes, with a Dirichlet prior on the class proportions and beta priors on the edge probabilities. \cite{Hofman2008} described a similar Bayesian approach for a special case of the SBM and suggested a variational approach to overcome the computational issues associated with maximizing over all possible class assignments. 

Bayesian methods for the SBM have barely been studied from a theoretical point of view, although
recent results for parameter recovery by \cite{Pati2015}, for detecting the number of communites by
\cite{Hayashi2016} and for an empirical Bayes approach to community detection by \cite{Suwan2016}
are encouraging. In this work, we provide theoretical results on community detection,
establishing that the Bayesian posterior mode is strongly consistent for the class labels if the
expected degree is at least of order $\log^2{n}$, where $n$ is the number of nodes. This is proven
by relating the posterior mode to the maximizer of the likelihood modularity of \cite{Bickel2009}.
The likelihood modularity has been claimed to be strongly consistent under the weaker assumption
that the expected degree is of larger order than $\log{n}$ \citep{Bickel2009, Zhao2012,
  Bickel2015}. However, their proof assumes that the likelihood modularity is globally Lipschitz,
while it is only locally so. The Bayesian method is based on a combination of likelihood and prior,
and for this reason the proof of our main theorem, Theorem \ref{thm:strongconsistency},
runs into a similar problem. We were able to resolve this only under the slightly stronger assumption that the expected degree is of larger
order than $(\log{n})^2$. The literature on other methods for community detection shows that
the order $\log n$ is sufficient for consistent detection. However, these
results are usually obtained under additional assumptions such as 
a restriction to two classes or an ordering of the connection probabilities,
and their implications for the likelihood or Bayesian modularities is unclear.
We discuss this and the relevant literature further following the statement of our main result in
Section \ref{sec:consistency}.

This paper is organized as follows. We introduce the SBM and the associated notation in Section~\ref{sec:SBM}. Our main results are in Section~\ref{sec:bayesian}, where we describe the prior and the link with the likelihood modularity, present the consistency results and discuss the underlying assumptions, especially those on the expected degree. The method is illustrated on a data set in Section~\ref{sec:application}, and we conclude with a Discussion in Section~\ref{sec:discussion}. All proofs are given in the Appendix. 

\section{\label{sec:SBM}The Stochastic Block Model}
We introduce the notation and generative model for the SBM with $K \in \{1,2,\ldots\}$
classes. Consider an undirected random graph with $n$ nodes, numbered $1, 2, \ldots , n$, and edges
encoded by the $n \times n$ symmetric adjacency matrix $(A_{ij})$, with entries in $\{0,1\}$. Thus $A_{ij}=A_{ji}$ is equal to
1 or 0 if the nodes $i$ and $j$ are or are not connected by an edge, respectively.
Self-loops are not allowed, so $A_{ii} = 0$ for $i = 1, \ldots, n$. The generative model for the
random graph is:
\begin{enumerate}
\item The nodes are randomly labeled with i.i.d. variables $Z_1, \ldots, Z_n$, taking
values in a finite set $\{1, . . . , K \}$, according to probabilities $\pi = (\pi_1,\ldots,\pi_K)$.
\item Given $Z = (Z_1, \ldots, Z_n)$, the edges are independently generated as Bernoulli variables with $\mathbb{P}(A_{ij} = 1 \mid Z ) = P_{Z_i, Z_j}$, for $i < j$, for a given $K \times K$ symmetric matrix $P = (P_{ab})$.
\end{enumerate}
The probability vector $\pi$ is considered fixed, but unknown.  Although this is not visible in the notation, the matrix $P$ may change with $n$, a case of particular interest being that $P$ tends to zero, which gives a sparse graph. The order of magnitude of $\|P\|_\infty = \max_{a,b} P_{ab}$ is the same as the order of magnitude of $\rho_n = \sum_{a,b} \pi_a\pi_b P_{ab}$, the probability of there being an edge between two randomly selected nodes. The \emph{expected degree} of a randomly
selected node is $\lambda_n = (n-1)\rho_n$, and twice the expected total number of edges in the network is $\mu_n = n(n-1)\rho_n$. 

The likelihood for the model is given by
\begin{equation}
\prod_{i<j} P_{Z_iZ_j}^{A_{ij}}(1-P_{Z_iZ_j})^{1-A_{ij}}\prod_{i}\pi_{Z_i}
=\prod_{a\le b} P_{ab}^{O_{ab}(Z)}(1-P_{ab})^{n_{ab}(Z)-O_{ab}(Z)}\prod_a\pi_a^{n_a(Z)},
\label{eq:likelihood}
\end{equation}
where $O_{ab}(Z)$ is the number of edges between nodes labelled $a$ and $b$ by the labelling $Z$,
$n_{ab}(Z)$ is the maximum number of edges that can be created between nodes labelled $a$ and $b$,
and $n_a(Z)$ is the number of nodes labelled $a$, and $a$ and $b$ range over $\{1,2,\ldots, K\}$.

More formally, for a given labelling $e = (e_1, \ldots, e_n) \in \{1, \ldots, K\}^n$ of nodes, 
and class labels $a, b \in \{1, \ldots, K\}$, we define
\begin{align*}
O_{ab}(e) &= 
\begin{cases}
\sum_{i, j} A_{ij}\1_{\{e_i = a, e_j = b\}}, & a \neq b,\\
\sum_{i < j} A_{ij}\1_{\{e_i = a, e_j = b\}}, & a = b,
\end{cases}\\
n_{ab}(e) &= \begin{cases}
n_a(e)n_b(e), & a \neq b,\\
\tfrac{1}{2}n_a(e)(n_a(e) -1), & a = b,
\end{cases}\\
n_{a}(e) &= \sum_{i = 1}^n \1_{\{e_i = a\}}.
\end{align*} 
Since the matrix $A$ is symmetric with zero diagonal by assumption,
for $a\not=b$ the variable $O_{ab}(e)$ can also be written as $\sum_{i<j}A_{ij}[\1_{\{e_i = a, e_j = b\}}+\1_{\{e_j = a, e_i = b\}}]$, which
explains the different appearances of the diagonal and off-diagonal entries.
The numbers $n_{ab}(e)$ are equal to the numbers $O_{ab}(e)$ when all $A_{ij}$ are equal to 1.
We collect the variables $O_{ab}(e)$ and $n_{ab}(e)$ in $K \times K$ matrices $O(e)$ and $n(e)$.

Now consider the $K \times K$ probability matrix $R(e, c)$ and $K$ probability vector $f(e)$ with entries
\begin{equation}\label{eq:defR}
R_{ab}(e, c) = \frac{1}{n} \sum_{i=1}^n \1_{\{e_i = a, c_i = b\}},  \qquad\quad f_a(e) = \frac{n_a(e)}{n}.
\end{equation}
The row sums of $R(e,c)$ are equal to $R(e,c)\1 = f(e)$, while the column sums are equal to $\1^TR(e,c) = f(c)^T$. Thus, the matrix $R(e,c)$ can be seen as a coupling of the marginal probability vectors $f(e)$ and $f(c)$. If $e = c$, then it is diagonal with diagonal $f(c) = f(e)$. More generally, the matrix can be viewed as measuring the discrepancy between labellings $e$ and $c$. This can be precisely measured as half the $L_1$-distance of $R(e,c)$ to its diagonal, as evidenced by Lemma \ref{lem:Rdiscrepancy}, which is noted in \cite{Bickel2009}. 

For a vector $v$ we denote by $\Diag(v)$ the diagonal matrix with diagonal $v$, and 
for a matrix $M$ we denote its diagonal by $\diag(M)$. 

\begin{lem} \label{lem:Rdiscrepancy}
For every labelling $c, e$ in the $K$-class stochastic block model: 
\begin{equation*}
\frac{1}{n} \sum_{i=1}^n \1_{\{c_i \neq e_i\}} = \tfrac{1}{2}\| \Diag(f(c)) - R(e,c)\|_1.
\end{equation*}
\end{lem}

\begin{proof}
The diagonal of $R(e,c)$ gives the fractions of labels on which $c$ and $e$ agree. Hence the left side of the lemma is $1 - \sum_a R_{aa}(e,c) = \sum_a (f_a(c) - R_{aa}(c))$ . The elements of both $K \times K$ matrices $\Diag (f(c)) $ and $R(e,c)$ can be viewed as probabilities that add up to 1. Thus the sum of the differences of the diagonal elements is minus the sum of the differences of the off-diagonal elements. Because $f_a(c) \geq R_{aa}(e,c)$ for every $a$, we have $\sum_a (f_a(c) - R_{aa}(e,c)) = \sum_a |f_a(c) - R_{aa}(e,c)|$. Similarly the off-diagonal elements of $\Diag (f(c))$, which are zero, are smaller than the off-diagonal elements of $R(e, c)$ and hence we can add absolute values. Thus the sum over the diagonal is half the sum of the absolute values of all terms in $\Diag (f(c)) - R(e,c)$.
\end{proof}

\section{\label{sec:bayesian}Bayesian Approach to Community Detection}
Our main results are presented in this section. We first discuss the choice of prior in Section \ref{sec:prior}, and define the estimator, in Section \ref{sec:BayesMod}. The resulting Bayesian modularity is closely related to the likelihood modularity of \cite{Bickel2009}. The relationship is clarified in Section \ref{sec:MLsim}. We briefly consider the issue of identifiability in the SBM in Section \ref{sec:ident}, and conclude with our main theorem on the strong consistency of the Bayesian modularity in Section \ref{sec:consistency}.

\subsection{\label{sec:prior}The prior}
We adopt the Bayesian approach of \cite{Nowicki2001}. We put prior distributions on
the parameters of the stochastic block model with $K$ known, 
the vector $\pi$ and the matrix $P$, yielding a joint probability distribution of $(A,Z,\pi,P)$.
Next we marginalize over $\pi$ and $P$ as in \cite{McDaid2013}, leading to a joint distribution of $(A,Z)$.
Finally we ``estimate'' the unobserved vector $Z$ by the posterior mode of the conditional distribution of $Z$ given $A$.
From a frequentist point of view this means that $Z$ is treated as a parameter of the problem, equipped with
a hierarchical prior that chooses first $\pi$ and then $Z$. Accordingly we shall change notation from $Z$ to $e$,
reserving $Z$ for the frequentist description of the stochastic block model in Section~\ref{sec:SBM}.

The prior on $\pi$ is a Dirichlet, and independently the $P_{ab}$ for $a\le b$ receive independent beta priors:
\begin{align*}
\pi &\sim {\rm Dir}(\alpha, \ldots, \alpha),\\
P_{ab} &\stackrel{i.i.d.}{\sim} {\rm Beta}(\beta_1, \beta_2), \quad 1 \leq a \leq b \leq K.
\end{align*} 
This is essentially the same set-up as in \cite{Nowicki2001} and \cite{McDaid2013}, except that we use a more flexible ${\rm Beta}(\beta_1, \beta_2)$ instead of a uniform prior on the $P_{ab}$. We assume $\alpha, \beta_1, \beta_2 > 0$.

We complete the Bayesian model by specifying class labels $e=(e_1,\ldots,e_n)$ and edges $A=(A_{ij}: i<j)$ through
\begin{align*}
e_i\mid \pi, P&\stackrel{i.i.d.}{\sim} \pi,\quad 1\le i\le n,\\
A_{ij}\mid \pi, P, e&\stackrel{ind.}{\sim} {\rm Bernoulli}(P_{e_i,e_j}),\quad 1\le i<j\le n.
\end{align*}
Abusing notation we write $p(e)$, $p(A\mid e)$  and $p(e\mid A)$ for  marginal and conditional probability
density functions.

\subsection{\label{sec:BayesMod}The Bayesian modularity}
The Bayesian estimator of the class labels will be the posterior mode, that is:
\begin{equation*}
\widehat e = \argmax_e p(e \mid A).
\end{equation*}
The posterior mode can be interpreted as a modularity-based estimator in the sense of \cite{Bickel2009}, 
in that it maximizes a function that only depends on the $O_{ab}(e)$ and the $n_a(e)$. 
This can be seen from the joint density of $(A,e)$, which is found by marginalizing the
likelihood \eqref{eq:likelihood} over $\pi$ and $P$. 
The conjugacy between the multinomial and Dirichlet distributions
gives the marginal density of the class assignment $e$ as:
\begin{align}
p(e) 
&=\int_{S_K}\prod_a \pi_a^{n_a(e)}\frac{\prod_a\pi_a^{\alpha-1}}{D(\alpha)}\,d\pi
=\frac{\Gamma(\alpha K)}{\Gamma(\alpha)^K \Gamma(n + \alpha K)} \prod_a\Gamma(n_a(e) + \alpha).\label{eq:marginaloverpi}
\end{align} 
Here the integral is relative to the Lebesgue measure on the $K$-dimensional unit simplex and
$D(\alpha)=\Gamma(\alpha)^K/\Gamma(K\alpha)$ is the norming constant for the Dirichlet density.
Similarly the conjugacy between the Bernoulli and Beta distributions gives the marginal conditional density of
$A$ given $e$ as:
\begin{align}
p(A \mid e) &=  \int_{[0,1]^{K(K+1)/2}}\prod_{a\le b} P_{ab}^{O_{ab}(e)}(1-P_{ab})^{n_{ab}(e)-O_{ab}(e)}
\prod_{a\le b}\frac{P_{ab}^{\beta_1 - 1}(1-P_{ab})^{\beta_2 - 1}}{B(\beta_1, \beta_2)}\,dP\nonumber\\
&=\prod_{a \leq b} \frac{1}{B(\beta_1, \beta_2)} B(O_{ab}(e) + \beta_1, n_{ab}(e) - O_{ab}(e) + \beta_2), \label{eq:marginaloverP}
\end{align}
where $B(x,y) = \Gamma(x)\Gamma(y)/\Gamma(x+y)$ is the beta-function. The joint density 
of $A$ and $e$ is given by the product of \eqref{eq:marginaloverpi} and \eqref{eq:marginaloverP}, 
and $n^{-2}$ times its logarithm is up to a constant that is free of $e$ equal to 
\begin{align*}
Q_B(e) = \frac{1}{n^2}\sum_{1 \leq a \leq b \leq K} \log B(\Oab + \beta_1, \nab - \Oab + \beta_2) 
+ \frac{1}{n^2}\sum_{a= 1}^K \log \Gamma(n_a(e) + \alpha).
\end{align*}
This is a  modularity in the sense of \cite{Bickel2009}, which we define as the \emph{Bayesian modularity}.
As $p(e \mid A$) is proportional to $p(e, A)$, the posterior mode is equal to the class 
assignment that maximizes the Bayesian modularity, so the Bayesian estimator is equal to:
\begin{equation}\label{eq:defestimator}
\widehat e = \argmax_e Q_B(e).
\end{equation}

\subsection{\label{sec:MLsim}Similarity to the likelihood modularity}
The Bayesian modularity $Q_B(e)$ consists of a two parts, originating from the likelihood and the
prior on the classification, respectively. The first part is close to the \emph{likelihood modularity}  given by
\begin{equation*}
Q_{ML}(e) = \frac{1}{n^2}\sum_{1\le a \leq b\le K} n_{ab}(e)\,\tau\Bigl(\frac{O_{ab}(e)}{n_{ab}(e)} \Bigr),
\end{equation*}
where $\tau(x) = x\log{x} + (1-x)\log(1-x)$. This criterion, obtained in \cite{Bickel2009}, 
results from replacing in the log conditional likelihood of $A$ given $e$ 
(the logarithm of \eqref{eq:likelihood} with $Z$ replaced by $e$ and discarding the term involving the
parameters $\pi_a$) the parameters $P_{ab}$ by
their maximum likelihood estimators $\hat P_{ab}=O_{ab}(e)/n_{ab}(e)$. 
In other words, the parameters are \emph{profiled out} rather than
integrated out as for the Bayesian modularity. The corresponding estimator 
$$\widehat e_{ML} = \argmax_e Q_{ML}(e)$$
is consistent, and hence one may hope that the Bayesian estimator can be proved consistent
by showing that the Bayesian and likelihood modularities are close. This will  indeed be
our line of approach, but the execution must be done with care. For instance,
the second, prior part of the Bayesian modularity does play a role in the proof of strong consistency,
although it is negligible when proving weak consistency.

The following lemma links the Bayesian and likelihood modularities.

\begin{lem}\label{lem:bayesstirling}
There exists a constant $C$ such that, for ${\cal E}=\{1,\ldots, K\}^n$ the set of all possible labellings:
\begin{align*}
&\max_{e\in {\cal E}}\Bigl| Q_B(e)-Q_{ML}(e)-Q_P(e)\Bigr|\le \frac{C\log n}{n^2}.
\end{align*}
for 
$$Q_P(e)= \frac{1}{n^2}\sum_{a: n_a + \floor{\alpha} \geq 2}  n_a(e)\log (n_a(e))-\frac1n.$$
Consequently 
$\max_{e\in {\cal E}}\bigl| Q_B(e)-Q_{ML}(e)\bigr|= {\cal O}\bigl(\log n/{n}\bigr)$.
\end{lem}

\subsection{\label{sec:ident}Identifiability and consistency}

A classification $\widehat e$ is said to be \emph{weakly consistent} if the fraction of misclassified nodes tends to zero (partial recovery), and \emph{strongly consistent} if the probability of misclassifying any of the nodes tends to zero (exact recovery). In defining consistency in a precise manner, the complication of the possible unidentifiability of the labels needs to be dealt with. From the observed data $A$ we can at best recover the partition of the $n$ nodes in the $K$ classes
with equal labels $Z_i$, but not the values $Z_1,\ldots, Z_n$ of the labels, in the set
$\{1,2,\ldots, K\}$, attached to the classes.  Thus consistency will be up to a permutation of labels.

To make this precise define, for a given permutation $(1,\ldots, K) \to \left(\sigma(1),\ldots, \sigma(K)\right)$, the
\emph{permutation matrix} $P_\sigma$ as the matrix with rows 
\begin{align*}
&e_{\sigma(1)}^T\\
&\ \vdots\\
&e_{\sigma(K)}^T,
\end{align*}
for $e_1,\ldots, e_K$ the unit vectors in $\mathbb{R}^K$. Then pre-multiplication
of a matrix by $P_\sigma$ permutes the rows, and post-multiplication by $P_\sigma^T$ the columns:
$P_\sigma R$ is the matrix with $j$th row equal to the $\sigma(j)$th row of $R$, and $RP_\sigma^T$ is the
matrix with $j$th column the $\sigma(j)$th column of $R$. Thus $P_\sigma R(e,Z)$ is the matrix
that would result if we would permute the labels of the classes of the assignment $e$,
and $P_\sigma PP_\sigma^T$ and $P_\sigma R(e,Z)P_\sigma^T$ are the matrices that would result if we would 
relabel the classes throughout. Since we cannot recover the labels, the matrix
$P_\sigma R(e,Z)$ is just as good or bad as $R(e,Z)$ for measuring discrepancy between a labelling $e$ and the true labelling $Z$; furthermore, nothing should change if we choose
different names for the classes. 

Thus, taking into account the unidentifiability of the labels, by Lemma \ref{lem:Rdiscrepancy}, 
an estimator $\widehat e$ is \emph{weakly consistent} if
\begin{equation*}
\|P_\sigma R(\widehat e, Z) - \Diag(f(Z))\|_1 \to 0,
\end{equation*}
for some permutation matrix $P_\sigma $. The classification $\widehat e$ is said to be \emph{strongly consistent} if 
\begin{equation*}
\mathbb{P}(P_\sigma R(\widehat e, Z) = \Diag (f(Z)) ) \to 1,
\end{equation*}
for some permutation matrix $P_\sigma$. 

The permutation matrix $P_\sigma$ is for large $n$ uniquely defined: if
  $\|({P_\sigma})_jR- \Diag(\pi)\|_1\le \min_a \pi_a$, for $j=1,2$, then $(P_\sigma)_1= (P_\sigma)_2$. This follows because the
assumption implies that $\|(P_\sigma)_1^{-1}\Diag(\pi)-(P_\sigma)_2^{-1}\Diag(\pi)\|_1\le 2\min_a\pi_a$, by the
triangle inequality and the fact that the $L_1$-norm is invariant under permutations.
Furthermore, for $P_\sigma=(P_\sigma)_2(P_\sigma)_1^{-1}$  the left side is $\|P_\sigma \Diag(\pi)-\Diag(\pi)\|_1$, which
 is at least two times the sum of the two smallest coordinates
of $\pi$ if $P_\sigma\not=I$.

A necessary requirement for consistency is that the classes can be recovered from the
likelihood, i.e.\ the model parameters must  be identifiable. 
If $\pi$ has strictly positive coordinates, so that all labels will appear in the
data eventually, then as explained in \cite{Bickel2009} an appropriate condition is that $P$ does not have two identical rows.
If $\pi_a=0$ for some $a$, then class $a$ will never be consumed; the
identifiability condition should then be imposed after deleting the $a$th column from $P$.
Thus, we call the pair $(P,\pi)$ \emph{identifiable} if the
rows of $P$ are different after removing the columns corresponding to zero coordinates of $\pi$.
Throughout we assume that $P$ is symmetric.

\subsection{\label{sec:consistency}Consistency results and assumptions}
We are now ready to present our results on consistency for the Bayesian maximum a posteriori (MAP) estimator \eqref{eq:defestimator}. Theorem \ref{thm:strongconsistency} shows strong consistency of the Bayesian estimator if $\lambda_n \gg (\log n)^2$. The proof rests on a proof of weak consistency under similar conditions, stated in the appendix as Theorem \ref{thm:weakconsistency}. 

Recall that $\rho_n = \sum_{a,b} \pi_a \pi_b P_{ab}$ is the probability of a new edge, and $\lambda_n = (n-1)\rho_n$ is the expected degree of a node. \\

\begin{thm}[strong consistency]
\label{thm:strongconsistency}
\begin{enumerate}[label = (\roman*)]
\item \label{thm:strongi}
If $(P,\pi)$ is fixed and identifiable with $0<P<1$ and $\pi>0$ then the MAP classifier $\widehat e
  =\arg\max_e Q_B(e)$  is strongly consistent. 
\item \label{thm:strongii}
If $P=\rho _nS$, where $(S,\pi)$ is fixed and identifiable with $S>0$ and $\pi>0$, then the MAP classifier 
$\widehat e=\arg\max_e Q_B(e)$  is strongly consistent if $\lambda_n \gg (\log n)^2$.
\end{enumerate}
\end{thm}

The theorem distinguishes two cases: \ref{thm:strongi} is the \emph{dense} case, while \ref{thm:strongii} is the \emph{sparse} case. The second is the most interesting of the two, as it touches on the question how much information is required to recover the underlying community structure. Much recent research effort has gone into determining detection and computational boundaries, in particular for special cases of the SBM with $K = 2$ (see e.g. \cite{Mossel2012}, \cite{Chen2014},  \cite{Abbe2014} and \cite{Zhang2015}).  

\emph{Weakly} consistent estimation of the class labels for an arbitrary, but known, number of classes is possible under the assumption $\lambda_n \gg \log{n}$, as this was shown to hold for spectral clustering by \cite{Lei2015}.  \emph{Strong} consistency of maximum likelihood was shown to hold in the special cases of planted bisection and planted clustering if $K = 2$ by \cite{Abbe2014, Chen2014}, again under the assumption $\lambda_n \gg \log{n}$. \cite{Gao2015} and \cite{Gao2016} achieve optimality in different senses, under assumptions on the average within-community and between-community edge probabilities; \cite{Gao2015} introduce a two-stage procedure which achieves the optimal proportion of misclassified nodes in a special case where $P_{ab}$ can only take two values, while \cite{Gao2016} obtain minimax rates for the proportion of misclassified nodes in the degree corrected SBM.

Strong consistency of the likelihood modularity for an arbitrary number of classes $K$ has been
claimed under the same assumption $\l_n \gg \log{n}$ \citep{Bickel2009}, and those results have been
extended to the degree-corrected SBM \citep{Zhao2012}. However, these results were obtained by application of
an abstract theorem to the special case of the likelihood modularity, which would require the function
$\tau(x) = x\log{x} + (1-x)\log(1-x)$, or the function $\sigma(x) = x\log{x}$, to be globally
Lipschitz. As $\tau$ and $\sigma$ are only locally Lipschitz, it is still unclear whether
$\l_n \gg \log{n}$ is a sufficient condition for either weakly or strongly consistent estimation by
maximum likelihood. From our proof of Theorem \ref{thm:strongconsistency}, which proceeds by
comparing the Bayesian modularity the likelihood modularity, it immediately follows that
$\l_n \gg (\log{n})^2$ is certainly sufficient. Given weak consistency the problem can be
reduced to a neighbourhood of the true parameter on which the Lipschitz condition is reasonable.
However, it is precisely our proof of weak consistency that needs the additional $\log n$ factor.

The Largest Gaps algorithm of \cite{Channarond2012} is strongly consistent provided that $\min_{a \neq b} |\sum_{k = 1}^K \alpha_k(P_{ak} - P_{bk})|$ is at least of order $\sqrt{\log n/n}$, implying that at least one of the $P_{ab}$ is of the same order, and thus $\lambda_n \gg \sqrt{n\log n}$. This much stronger condition is not surprising, as the Largest Gaps algorithm only uses the degree of a node and does not take into account any finer information on the group structure, such as the information contained in the $O_{ab}$.

To the best of our knowledge, for $K > 2$, it remains to be shown that $\lambda \gg \log n$ is sufficient for strong consistency of any community detection method for the general SBM. For the minimax rate for the proportion of misclustered nodes in community detection, when only classes of sizes proportional to $n$ are considered, a phase transition when going from the case $K = 2$ to $K \geq 3$ was observed by \cite{Zhang2015}. Their results show that if $K = 2$, communities of the same size are most difficult to distinguish, while if $K \geq 3$, small communities are harder to discover. This shift in the nature of the communities that are harder to detect may be what has been preventing a general strong consistency result under the assumption $\l_n \gg \log{n}$ so far.

\section{\label{sec:application}Application}
Some options for implementing the Bayesian modularity are given in Section \ref{sec:implementation}, after which the results of applying the Bayesian and likelihood modularities to the well-studied karate club data of \cite{Zachary1977} are discussed in Section \ref{sec:karate}.

\subsection{\label{sec:implementation}Implementation}
Two recent works explicitly discuss implementation of Bayesian methods for the SBM. \cite{McDaid2013} followed the approach of \cite{Nowicki2001} and added a Poisson prior on $K$. After marginalizing over $\pi$ and $P$, they employ an allocation sampler to sample from the joint density of $K$ and $z$ given $A$, and use the posterior mode to estimate $K$. Their algorithm can scale to networks with approximately ten thousand nodes and ten million edges. \cite{Come2014}, claiming that the algorithm of \cite{McDaid2013} suffers from poor mixing properties, propose a greedy inference algorithm for the same problem. For the karate club data in Section \ref{sec:karate}, the network was small enough that a tabu search \citep{Glover1989}, run for a number of different initial configurations, yielded good results. We used $\alpha = 1/2$ for the Dirichlet prior, and $\beta_1 = \beta_2 = 1/2$ for the beta prior. 

\subsection{\label{sec:karate}Karate club} 
\begin{figure}[t]
\begin{center}
\subfigure
{
\label{fig:K2}
\includegraphics[width = 0.48\textwidth]{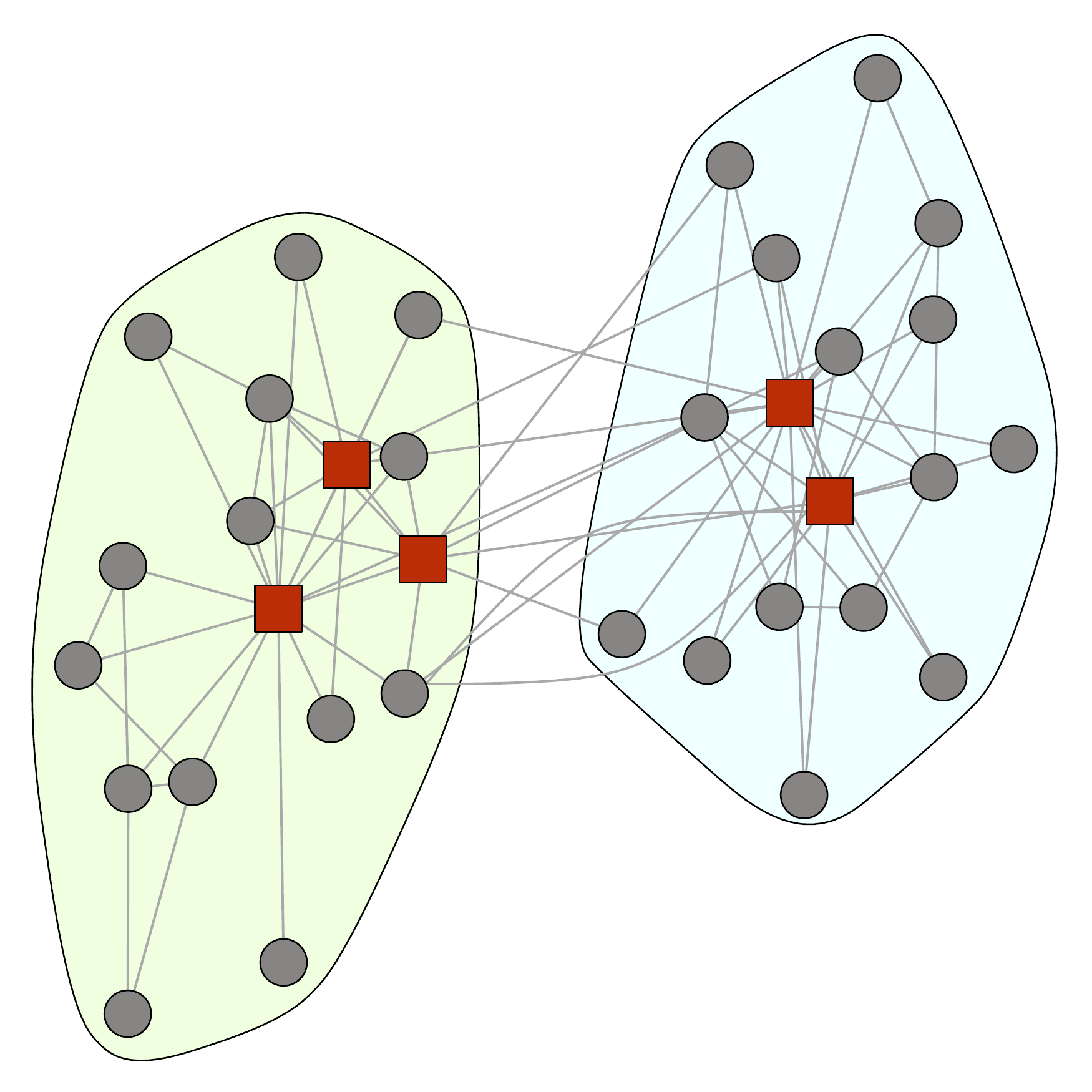}
}
\subfigure
{
\label{fig:K4}
\includegraphics[width = 0.48\textwidth]{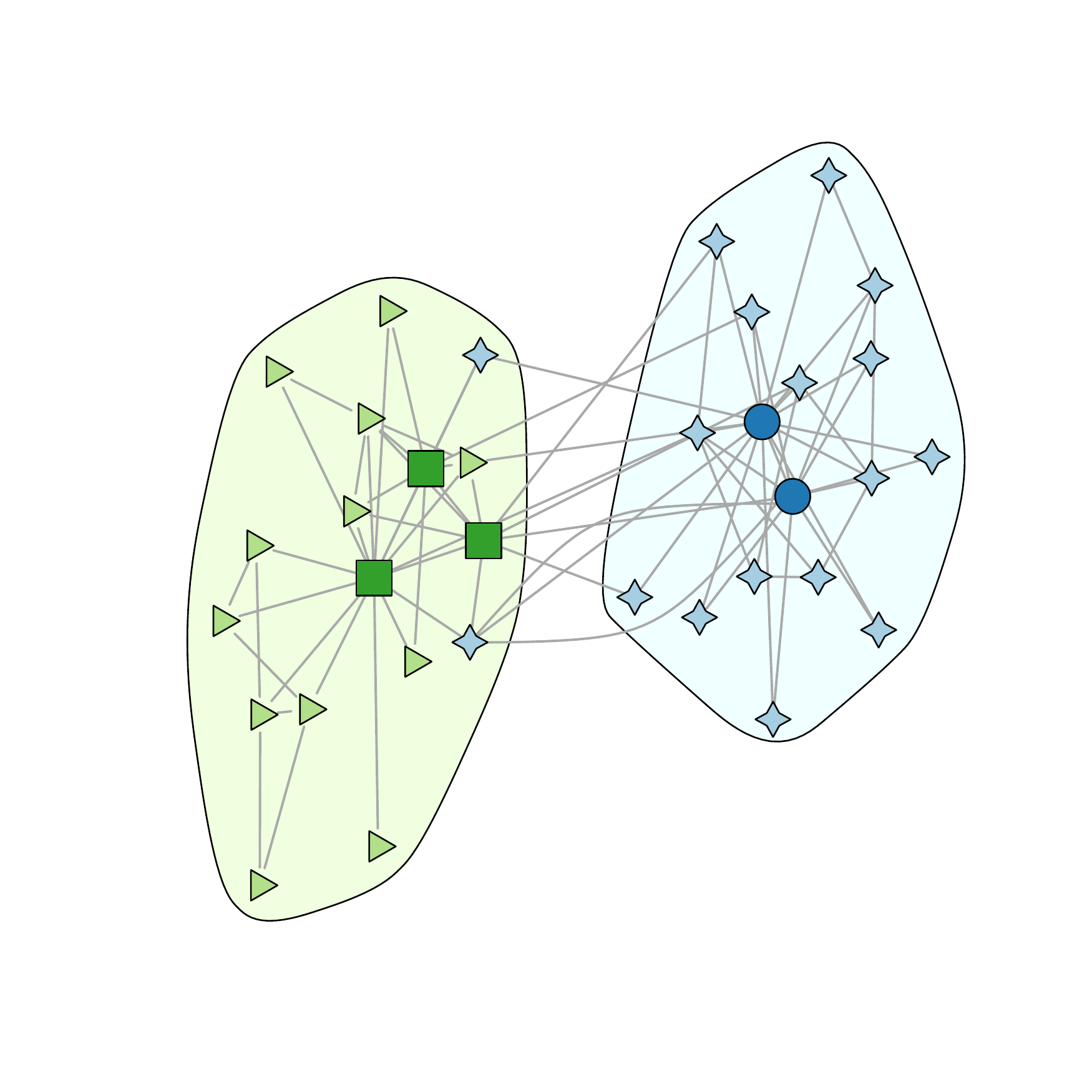}
}
\caption{Communities detected by the Bayesian modularity when $K = 2$ (left) and $K = 4$ (right), with $\alpha = \beta_1 = \beta_2 = 1/2$. The polygons contain the two groups the karate club was split into; the left one is Mr. Hi's club, the right one is the Officers' club. The shapes of the nodes represent the communities selected by the modularities. Figure made using the igraph package \citep{Csardi2006}.  }
\label{fig:karate}
\end{center} 
\end{figure}

\cite{Zachary1977} described a karate club which split into two clubs after a conflict over the price of the karate lessons. The new club was led by Mr. Hi, the karate teacher of the original club, while the remainder of the old club stayed under the former Officers' rule. The data consists of an adjacency matrix for those 34 individuals who interacted with other club members outside club meetings and classes. Each of these individuals' affiliations after the conflict is known. 

The communities selected by the Bayesian modularity for $K = 2$ and $K = 4$ are given in Figure \ref{fig:karate}. In both instances, the tabu search led to nearly the same solution for both the Bayesian and likelihood modularities, only differing at one node for $K = 4$, which is not surprising in light of Lemma \ref{lem:bayesstirling}. For $K = 2$, the results of \cite{Bickel2009} for this data set are recovered. For $K = 4$, the partition in Figure \ref{fig:karate} yields a higher value of the likelihood modularity than the partition into four classes found by \cite{Bickel2009}, and an even higher value is obtained by switching club member 20 to the second-largest class. This discrepancy is likely due to  the heuristic nature of the tabu search algorithm, and for the same reason, it may be the case that improvement over the partitions found by the Bayesian modularity in Figure \ref{fig:karate} are possible.

For $K = 2$, the communities found by the algorithms do not correspond in the slightest to the two karate clubs, instead grouping the nodes with the highest degrees, corresponding to Mr. Hi, the president of the original club, and their closest supporters, together. Incidentally, this partition is the same as the one returned by the Largest Gaps algorithm of \cite{Channarond2012}, which solely uses the degrees of the nodes and discards all other information.

These bad results are no reason to shelve the Bayesian and likelihood modularities, as there is no reason to believe that the two karate clubs form communities in the sense of the stochastic block model. Mr. Hi and the club's president are clear outliers within their groups, and neither of the algorithms were designed to be robust to such a phenomenon. The communities selected by the modularities are communities in the sense that they form connections within and between the groups in a similar fashion. This sense does not correspond to the social notion of a community in this setting. 

The results for four classes unify the social and stochastic senses of community. The prominent members of each of the new clubs are placed into two separate, small, communities. The other members are classified nearly perfectly, with two exceptions. However, one of those exceptional individuals is the only person described by \cite{Zachary1977} as being a supporter of the club's president before the split, who joined Mr. Hi's club, making this person's affiliation up for debate. The second is described as only a weak supporter of Mr. Hi. The increased number of communities allows for some outliers within the social communities, and leads to a more detailed understanding of the dynamics within both of the groups. We essentially recover the two communities, each with a core that is more connective than the remainder of te nodes.

\section{\label{sec:discussion}Discussion}
An advantage of Bayesian modelling is that it does not solely result in an estimator, but in a full posterior distribution. The posterior mode studied in this paper is but one aspect of the posterior, and its good behaviour in terms of consistency is encouraging. Further study into other aspects in the posterior may prove to be fruitful. One possible research direction would be to use the posterior to \emph{quantify uncertainty} in the estimate of the class labels. A second issue that may be resolved by the Bayesian approach is the question of estimating the number of classes, $K$. This remains an important open question, as noted by \cite{Bickel2009}, despite recent attempts (e.g. \cite{Saldana2014}, \cite{Chen2014-2} and \cite{Wang2015}). By introducing a prior on $K$, such as the Poisson-prior suggested by \cite{McDaid2013}, the number of communities $K$ can be detected by the posterior.

\appendix
\section{Proofs}
After stating some repeatedly used notation, this appendix starts with the proof of Theorem \ref{thm:weakconsistency}, which is a theorem on weak consistency of the Bayesian modularity. It is followed by a number of supporting Lemmas, after which we proceed to the proof of Theorem \ref{thm:strongconsistency}, and some additional supporting Lemmas.

We write $\diag(P)$ for the diagonal of $P$ if $P$ is a matrix, and $\Diag(f)$ for the 
diagonal matrix with diagonal $f$ if $f$ is a vector.

\subsection{Weak consistency}
The following quantities will be used in the course of multiple proofs.
The function $H_P$, with domain $K\times K$ probability matrices, is given by, for $\tau(u)=u\log u+(1-u)\log (1-u)$,
\begin{equation}\label{eq:HP}
H_P(R) = \frac{1}{2}\sum_{a, b} (R\1)_a(R\1)_b\, \tau\left(\frac{(RPR^T)_{ab}}{(R\1)_a(R\1)_b} \right).
\end{equation} 
For $\tau_0(u)=u\log(u)-u$, define
\begin{equation*}
G_P(R) =\frac{1}{2}\sum_{a, b}(R\1)_a(R\1)_b\,\tau_0\Bigl(\frac{(RPR^T)_{ab}}{(R\1)_a(R\1)_b}\Bigr).
\end{equation*}
The sums defining these functions are over all pairs $(a,b)$ with $1\le a,b\le K$, unlike
the sums defining the modularities $Q_B$ and $Q_{ML}$, which are restricted to $a\le b$.

\begin{thm}[weak consistency]
\label{thm:weakconsistency}
\begin{enumerate}[label = (\roman*)]
 \item \label{thm:weaki} If $(P,\pi)$ is fixed and identifiable, then the MAP classifier $\widehat
  e=\arg\max_z Q_B(e)$ is weakly consistent.
\item \label{thm:weakii} If $P=\rho_n S$ for $\rho_n \to 0$, and $(S,\pi)$ is fixed and identifiable, then the  MAP classifier $\widehat
  e=\arg\max_z Q_B(e)$  is weakly consistent provided $n\rho_n\gg (\log n)^2$.
  \end{enumerate}
\end{thm}

\begin{proof}
By Lemma~\ref{lem:bayesstirling} the Bayesian modularity $Q_B$ is equivalent to the likelihood
modularity $Q_{ML}$ up to order $(\log n)/n$. 
With the notation $\wt O_{ab}(e) = O_{ab}(e)$ if $a \neq b$, and $\wt O_{ab}(e) = 2O_{ab}(e)$ if $a = b$, 
the likelihood modularity is in turn equivalent up to the same order to 
\begin{equation}\label{eq:defLL}
\LL(e) = \frac1{2n^2}\sum_{a, b} n_a(e)n_b(e)\,\tau\Bigl(\frac{\wt O_{ab}(e)}{n_a(e)n_b(e)} \Bigr).
\end{equation}
Indeed the terms of $Q_{ML}(e)$ for $a <b$ are identical to the sums of the terms of 
$\LL(e)$ for $a<b$ and $a>b$, while for $a = b$ the terms of $Q_{ML}(e)$ and $\LL(e)$ differ
only subtly: the first uses $n_{aa}(e) = \tfrac{1}{2}n_a(e)(n_a(e)-1)$, where the second
uses $\tfrac{1}{2}n_a(e)^2$. Thus the difference is bounded in absolute value by
the sum over $a$ of (where $e$ is suppressed from the notation)
$$\Bigl|\frac {n_a^2}{2n^2}\tau\Bigl(\frac{\wt O_{aa}}{n_a^2} \Bigr)
-\frac{n_a\bigl(n_a-1)}{2n^2}\tau\Bigl(\frac{\wt O_{aa}}{n_a(n_a-1)} \Bigr)\Bigr|
\le \frac 1{2n}\|\tau\|_\infty+\frac{n_a^2}{2n^2}l\Bigl(\frac{\wt O_{aa}}{n_a^2(n_a-1)}\Bigr).$$
where $l(x)=x(1\vee \log (1/x))$, in view of Lemma~\ref{lem:boundtau}. 
We now use that $n_a l(u/n_a)\lesssim \log n_a\le \log n$, for $0\le u\le 1$.

Combining the preceding, we conclude that
\begin{equation*}
\eta_{n,1} := \max_e |\LL(e) - Q_B(e)| = \mathcal{O}\left(\frac{\log{n}}{n}\right).
\end{equation*}
Since $Q_B(\widehat e) \geq Q_B(Z)$, by the definition of $\widehat e$, it follows that
$\LL(\widehat e)-\LL(Z)\ge -2\eta_{n,1}$. 
The next step is to replace $\LL$ in this equality by an asymptotic value. 

For $x$ equal to a big multiple of $(\|P\|_\infty^{1/2}\vee n^{-1/2})/n^{1/2}$, the right side of Lemma~\ref{lem:O} tends
to zero and hence $\max_e\bigl\|\wt O(e)-\E\bigl(\wt O(e) \mid  Z\bigr)\bigr\|_\infty/n^2$ is of this order 
in probability. We also have, by Lemma \ref{lem:Oconditionalexpectation}:
 \begin{equation*}
 \max_e \Bigl\| \frac1{n^2}\E\bigl(\wt O(e) \mid Z\bigr) - R(e, Z)P R(e,Z)^T\Bigr\|_\infty 
= \max_e \frac{1}{n}\bigl\|\Diag(R(e,Z)) \diag(P)\bigr\|_\infty \to 0,
 \end{equation*}
 as each entry of $\Diag(R(e,Z))\diag(P)$ is bounded above by one. 
By Lemma~\ref{lem:boundtau}, $\bigl|v\tau(x/v)-v\tau(y/v)\bigr|\le l(|x-y|)$, uniformly in $v\in [0,1]$, where 
$l(x) = x(1 \vee \log(1/x))$. It follows that 
\begin{equation*}
\eta_{n,2}:=\max_{e}\bigl|\LL(e)-L(e)\bigr|=o_P\Bigl(l\Bigl(\frac{\|P\|_\infty^{1/2}\vee n^{-1/2}}{n^{1/2}}\Bigr)\Bigr),
\end{equation*}
for
\begin{equation*}
L(e) = \frac{1}{2}\sum_{a, b} f_a(e)f_b(e)\,\tau\Bigl( \frac{(R(e,Z)PR(e,Z)^T)_{ab}}{f_a(e)f_b(e)} \Bigr).
\end{equation*}
Combining this with the preceding paragraph, we conclude that
$L(\widehat e) \geq L(Z) - 2(\eta_{n,1} + \eta_{n,2})$.

Proof of \ref{thm:weaki}. 
For given $\delta>0$, let $\mathcal{R}_\delta$ be the set of all probability matrices $R$ with 
\begin{equation*}
\min_{P_\sigma} \bigl\|P_\sigma R-\Diag(R^T\1)\bigr\|_1\ge \delta, \qquad\text{ and } \qquad\min_{a: \pi_a>0} (R^T\1)_a\ge \delta.
\end{equation*}
Here the minimum is taken over the (finite) set of all permutation matrices $P_\sigma$ on $K$
labels. Furthermore, set
\begin{equation*} 
\eta :=\inf_{R\in\mathcal{R}_\delta}\bigr[H_P\bigl(\Diag(R^T\1)\bigr)-H_P(R)\bigl],
\end{equation*}
where $H_P$ is as defined in $\eqref{eq:HP}$. Because $\mathcal{R}_\delta$ is compact and the maps $R\mapsto H_P(R)$ and $R\mapsto \Diag(R^T\1)$ are continuous,
the infimum in the display is assumed for some $R\in \mathcal{R}_\delta$. Because no $R\in\mathcal{R}_\delta$ can be 
transformed into a diagonal element by permuting rows and every $R\in \mathcal{R}_\delta$ has 
a nonzero element in every column $a$ with $\pi_a>0$, Lemma~\ref{lem:maximality} shows that  $\eta_n>0$.

Because $L(e) = H_P(R(e,Z))$ for every $e$, and $R(Z,Z) =\Diag(f(Z)) =\Diag(R(\widehat e, Z)^T\1)$, we conclude that
\begin{equation*}
H_P( \Diag(R(\widehat e, Z)^T\1)) - H_P(R(\widehat e, Z) ) \leq 2(\eta_{n,1} + \eta_{n,2}) .
\end{equation*}
If $2(\eta_{n,1} + \eta_{n,2})$ is smaller than $\eta_n$, then it follows that $R(\widehat e, Z)$ cannot be contained in $\mathcal{R}_\delta$. Since $R(\widehat e, Z)^T\1 = f(Z) \stackrel{P}{\to} \pi$, by the law of large numbers, for sufficiently small $\delta > 0$ this must be because $R(\widehat e, Z)$ fails the first requirement defining $\mathcal{R}_\delta$. That is, $\|P_\sigma R(\widehat e, Z) - \Diag(f(Z))\|_1 \leq \delta$ for some permutation matrix $P_\sigma$. As this is true eventually for any $\delta > 0$, it follows that $\min_{P_\sigma} \|P_\sigma R(\widehat e, Z) - \Diag (\pi)\|_1 \stackrel{P}{\to} 0$. 

Proof of \ref{thm:weakii}. In view of Lemma~\ref{lem:maximality2}, the number $\eta=\eta_n$, which now depends on $n$,
 is now bounded below by $\rho_n$ times a positive number
that depends on $(S,\pi)$. The preceding argument goes through provided
$\eta_{n,1} + \eta_{n,2}$ is of smaller order than $\eta_n$. 
This leads to $l\bigl(\sqrt{\rho_n/n}\bigr) + \log(n)/n \ll \rho_n$, or
$(\rho_n/n)\log^2 \bigl(n/(\rho_n\|S\|_\infty)\bigr)\ll \rho_n^2$.
\end{proof}

\begin{lem}
\label{lem:O}
Let $\wt O_{ab}(e) = O_{ab}(e)$ if $a \neq b$, and $\wt O_{ab}(e) = 2O_{ab}(e)$ if $a = b$. For any $x>0$, 
$$\mathbb{P}\Bigl(\max_e\bigl\|\wt O(e)-\E\bigl(\wt O(e) \mid Z\bigr)\bigr\|_\infty> xn^2\Bigr)
\le 2 K^{n+2} e^{-x^2n^2/(8\|P\|_\infty+4x/3)}.$$
\end{lem}

\begin{proof}
This Lemma is adapted from Lemma 1.1 in \cite{Bickel2009}. There are $K^n$ possible values of $e$ and $\|\cdot\|_\infty$ is the maximum of
the $K^2$ entries in the matrix. We use the union bound to pull these maxima out of
the probability, giving the factor $K^{n+2}$ on the right. Next it suffices to
bound the tail probability of each variable 
\begin{equation*}
\wt O_{ab}(e)-\E\bigl(\wt O_{ab}(e) \mid Z\bigr)=\sum_{i,j}\bigl(A_{ij}-\E(A_{ij} \mid Z)\bigr)(
\1\{e_i=a,e_j=b\} + \1\{e_i = b, e_j = a\}).
\end{equation*}
The  $n_{ab}(e)$ variables in this sum are conditionally independent given $Z$,
take values in $[-2,2]$, 
and have conditional mean zero given $Z$ and conditional variance 
bounded by  $4\var (A_{ij} \mid Z)\le 4P_{Z_iZ_j}(1-P_{Z_iZ_j})\le 4\|P\|_\infty$.
Thus we can apply Bernstein's inequality to find that
\begin{equation*}
\mathbb{P}\Bigl(\bigl|\wt O_{ab}(e)-\E\bigl(\wt O_{ab}(e) \mid Z\bigr)\bigr|>xn^2\Bigr)
\le 2e^{-x^2n^4/(8n_{ab}(e)\|P\|_\infty+4xn^2/3)}.
\end{equation*}
Finally we use the crude bound $n_{ab}(e)\le n^2$ and cancel one factor $n^2$. 
\end{proof}

\begin{lem}\label{lem:Oconditionalexpectation}
Define $\wt O_{ab}(e) = O_{ab}(e)$ if $a \neq b$, and $\wt O_{ab}(e) = 2O_{ab}(e)$ if $a = b$. Then,
for $R(e, Z)$ as defined in \eqref{eq:defR},
\begin{equation*}
\E(\wt O_{ab} \mid Z ) = n^2 R(e, Z) P R(e, Z)^T - n \Diag( R(e,Z) \diag(P)).
\end{equation*}
\end{lem}

\begin{proof}
A similar expression, not taking into account the absence of self-loops, appears in \cite{Bickel2009}.
\begin{align*}
&\E(\wt O_{ab}(e)  \mid Z = c) = \sum_{i \neq j} P_{c_ic_j}\1\{e_i = a, e_j = b\}\\
&\quad= \sum_{a', b'} P_{a'b'} \sum_{i \neq j} \1\{c_i = a', c_j=b'\}\1\{e_i = a, e_j = b\}\\ 
&\quad= \sum_{a', b'} P_{a'b'} \sum_{i, j} \1\{c_i = a', c_j=b'\}\1\{e_i = a, e_j = b\} - \d_{ab}\sum_{a'} P_{a'a'} \1\{c_i = a'\}\1\{e_i=a\}\\ 
&\quad= n^2 \sum_{a', b'} P_{a'b'} R_{aa'}(e,c)R_{bb'}(e, c) - \d_{ab} n \sum_{a'} P_{a'a'} R_{aa'}(e, c).
\end{align*}
\end{proof}

\begin{lem} \label{lem:boundtau}
The function $\tau: [0,1]\to \mathbb{R}$ satisfies $|\tau(x)-\tau(y)|\le l(|x-y|)$, for
$l(x)= 2x(1\vee \log(1/x))$.
\end{lem}
\begin{proof}
Write the difference between $x\log{x}$ and $y\log{y}$ as $|\int_x^y(1+\log s)\,ds|$. The
  function $s\mapsto 1 + \log s$ is strictly increasing on $[0,1]$ from $-\infty$ to 1 and changes sign at
  $s=e^{-1}$. Therefore the absolute integral is bounded above by the maximum of 
\begin{align*}
-\int_0^{|x-y|\wedge  e^{-1}}(1+\log s)\,ds&=-(|x-y|\wedge e^{-1})\log {|x-y|\wedge e^{-1}}\\
\noalign{\noindent \textrm{and}}
\int_{1-|x-y|\vee  e^{-1}}^1(1+\log s)\,ds&\le |x-y|.
\end{align*}
\end{proof}

\noindent\textbf{Proof of Lemma \ref{lem:bayesstirling}}

\begin{proof}
The second assertion of the lemma follows from the first and the fact that $\max _e Q_P(e)\lesssim (\log n)/n$.
It suffices to prove the first assertion.

Recall that the Bayesian modularity is given by 
\begin{equation}\label{eq:Bmodrecall}
n^2 Q_B(e) = \sum_{a \leq b} \log B\left( O_{ab}(e) + \tfrac{1}{2},  n_{ab}(e) -  O_{ab}(e) + \tfrac{1}{2}\right) 
+ \sum_a \log \Gamma(n_a(e) + \alpha).
\end{equation} 
We shall show that the first sum on the right is equivalent to $Q_{ML}(e)$, and the second sum 
is equivalent to $Q_P(e)$. We show this by comparing
the sums defining the various modularities term by term. 
For clarity we shall suppress the argument $e$. We will repeatedly use the following bound from \citep{Robbins1955}: for $n \in \mathbb{N}_{\geq 1}$,
\begin{equation}\label{eq:stirling}
\Gamma(n+1) =  \sqrt{2\pi}n^{n+1/2}e^{-n}e^{a_n},
\end{equation}
with $(12n+1)^{-1} \leq a_n \leq(12n)^{-1}$, as well as the fact that $\Gamma(s)$ is monotone increasing for $s \geq 3/2$. In addition, we will bound remainder terms by using the inequality $x\log((x+c)/x) \leq c$ for $c \geq 0$ and the fact that $x\log((x-1)/x)$ is bounded for $x > 1$.

\textbf{First sum of \eqref{eq:Bmodrecall}.}\\
\emph{Upper bound, case 1: $O_{ab} \neq 0$ and $n_{ab} \neq O_{ab}$}\\
We apply \eqref{eq:stirling}:
\begin{align*}
\log &B(O_{ab} + \beta_1, n_{ab} - O_{ab} + \beta_2) 
\leq \log \frac{\Gamma(O_{ab} + \floor{\beta_1} + 1)\Gamma(n_{ab}-O_{ab} + \floor{\beta_2}+1) }{\Gamma(n_{ab} + \floor{\beta_1+\beta_2})}\\
&= O_{ab}\log\left(\frac{O_{ab}+\floor{\beta_1}}{n_{ab} + \floor{\beta_1+\beta_2} - 1} \right)
+ (n_{ab} - O_{ab})\log\left(\frac{n_{ab}-O_{ab}+\floor{\beta_2}}{n_{ab}+\floor{\beta_1+\beta_2}-1} \right)\\
&\quad + (\floor{\beta_1}+1/2)\log(O_{ab}+\floor{\beta_1}) + (\floor{\beta_2}+1/2)\log(n_{ab}-O_{ab}+\floor{\beta_2})\\ 
&\quad - (\floor{\beta_1+\beta_2}-1/2)\log(n_{ab}+\floor{\beta_1+\beta_2}-1) 
+ \log\sqrt{2\pi}  - \floor{\beta_1}  - \floor{\beta_2}+\floor{\beta_1+\beta_2}-1\\
&\quad + \alpha_{ab} + \beta_{ab} - \gamma_{ab} ,
\end{align*}
where $\alpha_{ab}, \beta_{ab}$ and $\gamma_{ab}$ are bounded by constants. By the inequality $x\log( (x+c)/x) \leq c$ for $c \geq 0$, and the fact that $x\log((x-1)/x)$ is bounded for $x > 1$, we find the upper bound:
$$\log B(O_{ab} + \beta_1, n_{ab} - O_{ab} + \beta_2) 
\leq n_{ab}\tau\left(\frac{O_{ab}}{n_{ab}}\right) + O(\log{n_{ab}}). $$

\emph{Upper bound, case 2: $n_{ab} = 1$ and $O_{ab} = 0$ or $n_{ab} = O_{ab}$, or $n_{ab} = 0$}\\
In both cases, the corresponding term of the likelihood modularity vanishes, whereas the contribution of the Bayesian modularity is either $\log B(1+\beta_1, \beta_2)$, $\log(\beta_1, 1 + \beta_2)$, or $\log B(\beta_1, \beta_2)$. 

\emph{Upper bound, case 3: $n_{ab} \geq 2$ and $O_{ab} = 0$ or $n_{ab} = O_{ab}$}\\
Again, the corresponding term of the likelihood modularity vanishes. We show the computations for the case $n_{ab} = O_{ab}$; for the case $O_{ab} = 0$, switch $\beta_1$ and $\beta_2$. By \eqref{eq:stirling}:
\begin{align*}
\log &B(O_{ab} + \beta_1, n_{ab} - O_{ab} + \beta_2) 
= \log B(n_{ab} + \beta_1, \beta_2) \leq \log \frac{\Gamma(n_{ab} + \floor{\beta_1}+1)\Gamma(\beta_2)}{\Gamma(n_{ab}+\floor{\beta_1+\beta_2})}\\
&= (n_{ab} + \floor{\beta_1})\log\left(\frac{n_{ab}+\floor{\beta_1}}{n_{ab} + \floor{\beta_1 + \beta_2}} \right)
+ (1/2)\log(n_{ab}+\floor{\beta_1})\\ 
&\quad - (\floor{\beta_1+\beta_2}+1/2)\log(n_{ab}+\floor{\beta_1+\beta_2}) + \log \Gamma(\beta_2)+\floor{\beta_1+\beta_2}-1+\delta_{ab}-\epsilon_{ab},
\end{align*}
where $\delta_{ab}$ and $\epsilon_{ab}$ are bounded by constants. Arguing as before, the first term is bounded, while the remainder is of order $\log(n_{ab})$. A lower bound is found analogously. \\ 

\emph{Lower bound}
The computations for the lower bound are completely analogous, except that we require $O_{ab} + \beta_1 \geq 2$ and $n_{ab} - O_{ab} + \beta_2 \geq 2$. We study four cases. The cases (1) $O_{ab} \geq 2$ and $n_{ab} - O_{ab} \geq 2$, (2) $n_{ab} = 0$ and (3) $n_{ab} > 0$ and $n_{ab} = O_{ab}$ or $O_{ab} = 0$ are similar to cases 1, 2 and 3 respectively of the upper bound. The fourth case is $n_{ab}-O_{ab} = 1$ and $O_{ab} \geq 2$, or $O_{ab} = 1$ and $n_{ab}-O_{ab} \geq 1$. In both instances, the likelihood modularity is equality to a bounded term minus $\log{n_{ab}}$. By similar calculations as before, the Bayesian modularity is of the order $\log{n_{ab}}$ as well.\\

\emph{Conclusion}
We find:
\begin{equation*}
\sum_{a\leq b}\log B(O_{ab} + \beta_1, n_{ab} - O_{ab} + \beta_2) = \sum_{a \leq b} n_{ab}\tau\left(\frac{O_{ab}}{n_{ab}}\right) + O(\log{n}).
\end{equation*}

\textbf{Second sum of \eqref{eq:Bmodrecall}.}\\
We consider three cases. If $n_a + \floor{\alpha} = 0$, then $\alpha > 0$, implies $n_a = 0$, in which case $\log \Gamma(n_a + \alpha) = \log \Gamma(\alpha)$, which is bounded. In case $n_a + \floor{\alpha} = 1$, the term $\log \Gamma(n_a + \alpha)$ is equal to either $\log \Gamma(1 + \alpha)$ or $\log \Gamma(\alpha)$ and thus bounded as well. For the case $n_a + \floor{\alpha} \geq 2$, we study the upper bound $ \Gamma(n_a + \alpha) \leq \Gamma(n_a + \floor{\alpha} + 1)$ and the lower bound $\Gamma(n_a + \alpha) \geq \Gamma(n_a + \floor{\alpha})$. By applying \eqref{eq:stirling} in both cases, we conclude: 

$$\sum_a \log \Gamma(n_a + \alpha) = \sum_{a: n_a + \floor{\alpha} \geq 2} n_a \log n_a - n + O(\log n). $$

\end{proof}

\begin{lem}
\label{lem:maximality}
For any probability matrix $R$,
\begin{equation}
\label{EqDiagRBigger}
H_P(R)\le H_P(\Diag(R^T\1)\bigr).
\end{equation}
Furthermore, if $(P,\pi)$ is identifiable and the columns of $R$ corresponding to positive
coordinates of $\pi$ are not identically zero,  then the inequality is strict unless $P_\sigma R$ is a diagonal matrix
for some permutation matrix $P_\sigma$. 
\end{lem}

\begin{proof}
This Lemma is related to the proof that the likelihood modularity is consistent given in \cite{Bickel2009}. This proof however rests on their incorrect Lemma 3.1, and thus we provide full details on how the argument can be adapted to avoid the use of their Lemma 3.1 altogether.

For $R$ a diagonal matrix the numbers $(RPR^T)_{ab}/(R\1)_a(R\1)_b$ reduce to
$P_{ab}$. Consequently, by the definition of $H_P$,
\begin{equation}
\label{EqHforDiag}
H_P\bigl(\Diag(f)\bigr)=\sum_{a , b}f_af_b\, \tau(P_{ab}).
\end{equation}
For a general matrix $R$, by inserting the definition of $\tau$,
\begin{align*}
H_P(R)&=\sum_{a , b}(RPR^T)_{ab}\log\frac{(RPR^T)_{ab}}{(R\1)_a(R\1)_b}\\
&\qquad\qquad+\sum_{a , b}\bigl((R\1)_a(R\1)_b-(RPR^T)_{ab}\bigr)\log\Bigl(1-\frac{(RPR^T)_{ab}}{(R\1)_a(R\1)_b}\Bigr).
\end{align*}
Because $(R\1)_a(R\1)_b-(RPR^T)_{ab}=(R(1-P)R^T)_{ab}$, with $1$ the $(K\times K)$-matrix with all
coordinates equal to 1, we can rewrite this as 
\begin{align*}
&\sum_{a , b}\sum_{a',b'}R_{aa'}R_{bb'}\biggl[P_{a'b'}\log\frac{(RPR^T)_{ab}}{(R\1)_a(R\1)_b}
+(1-P_{a'b'})\log\Bigl(1-\frac{(RPR^T)_{ab}}{(R\1)_a(R\1)_b}\Bigr)\biggr].
\end{align*}
By the information inequality for two-point measures, the expressions in square brackets 
becomes bigger when $(RPR^T)_{ab}/(R\1)_a(R\1)_b$ is replaced by $P_{a'b'}$, with a strict increase
unless these two numbers are equal. After making this substitution the terms in square brackets
becomes $\tau(P_{a'b'})$, and we can exchange the order of the
two (double) sums and perform the sum on $(a,b)$ to write the resulting expression as
$$\sum_{a',b'}(R^T\1)_{a'}(R^T\1)_{b'}\tau(P_{a'b'})=H_P\bigl(\Diag(R^T\1)\bigr).$$
This proves the first assertion \eqref{EqDiagRBigger} of the lemma.

If $R$ attains equality, then also for every permutation matrix $P_\sigma$, by the equality
$H_P(P_\sigma R)= H_P(R)$
and the fact that $(P_\sigma R)^T\1=R^T\1$, we have
\begin{equation}
\label{EqHulp}
H_P(P_\sigma R)=H_P\bigl(\Diag((P_\sigma R)^T\1)\bigr).
\end{equation}
We shall show that if $R$ satisfies this equality and $P_\sigma R$ has a positive diagonal,
then $P_\sigma R$ is in fact diagonal. Furthermore, we shall show that there exists $P_\sigma$ such
that $P_\sigma R$ has a positive diagonal.

Fix some $(P_\sigma)_m$ that maximizes the number of positive diagonal elements of $P_\sigma R$ over all
permutation matrices $P_\sigma$, and denote $\bar R=(P_\sigma)_mR$. Because the information inequality is
strict, the preceding argument shows that \eqref{EqHulp} can be true for $P_\sigma=(P_\sigma)_m$ (giving $P_\sigma R=\bar
R$) only if 
\begin{equation}
\label{EqNoLossR}
P_{a'b'}=\frac{(\bar RP\bar R^T)_{ab}}{(\bar R\1)_a(\bar R\1)_b}, \qquad \text{whenever
} \bar R_{aa'}\bar R_{bb'}>0.
\end{equation}
Denote the matrix on the right of the equality by $Q$. 

If $\bar R$ has a completely positive
diagonal, then we can choose $a=a'$ and $b=b'$ and find from
equation \eqref{EqNoLossR}, that $P_{ab}=Q_{ab}$, for every $a,b$.
If also $\bar R_{aa'}>0$, then we can also choose $b=b'$ and find that $P_{a'b}=Q_{ab}$, for every $b$.
Thus the $a$th and $a'$th rows of $P$ are identical. Since all rows of $P$
are different by assumption, it follows that no  $a\not=a'$ with $\bar R_{aa'}>0$ exists.

If $\bar R$ does not have a fully positive diagonal, then the submatrix of $\bar R$ obtained by
deleting the rows and columns corresponding to positive diagonal elements must be the zero matrix,
since otherwise we might permute the remaining rows and create an additional nonzero diagonal
element, contradicting that $(P_\sigma)_m$ already maximized this number.
If $I$ and $I^c$ are the sets of indices of zero and nonzero diagonal elements, then 
the preceding observation is that $\bar R_{ij}$ is  zero for every $i,j\in I$. If $\pi>0$, then we
need to consider only $R$ with nonzero columns. For  $i\in I$ a nonzero element in the $i$th column of $\bar R$
must be located in the rows with label in $I^c$: for every $i\in I$ there exists $k_i\in I^c$ with $\bar
R_{k_ii}>0$.
Then, for $i,j\in I$,
\begin{itemize}
\item[(1)] for $a=k_i$, $b=k_j$, $a'=i$, $b'=j$, equation \eqref{EqNoLossR} implies $Q_{k_ik_j}=P_{ij}$.
\item[(2)] for $a=k_i$, $b\in I^c$, $a'=i$, $b'=b$, equation \eqref{EqNoLossR} implies
  $Q_{k_ib}=P_{ib}$.
\item[(3)] for $a=k_i$, $b\in I^c$, $a'=k_i$, $b'=b$, equation \eqref{EqNoLossR} implies $Q_{k_ib}=P_{k_ib}$.
\end{itemize}
We combine these three assertions to conclude that, for $a,i\in I$ and $b\in I^c$,
\begin{align*}
P_{ai}&=P_{ia}\stackrel{(1)}{=}Q_{k_ik_a}\stackrel{(2)}{=}P_{ik_a}=P_{k_ai},\\
P_{ab}&\stackrel{(2)}{=}Q_{k_ab}\stackrel{(3)}=P_{k_ab}.
\end{align*}
Together these imply that the $a$th and the $k_a$th row of $P$ are equal. Since by assumption
they are not (if $\pi>0$), this case can actually not exist (i.e.\ $k=0$). 

Finally if $\pi_a=0$ for some $a$, then we follow the same argument, but we match only
every column $i\in I$ with $\pi_i>0$ to a row $k_i\in I^c$. By the assumption on $R$ such $k_i$
exist, and the construction results in two rows of $P$ that are identical in the coordinates with $\pi_a>0$.\end{proof}

\begin{lem}
\label{lem:maximality2}
For any fixed $(K\times K)$-matrix $P$ with elements in $[0,1]$, uniformly in probability matrices $R$, as $\rho_n\ra0$,
\begin{equation}
\label{EqApproximateHbyG}
\frac{1}{\rho_n}\Bigl( H_{\rho_nP}(\Diag(R^T\1)\bigr)-H_{\rho_n P}(R)\Bigr)\ra G_P(\Diag(R^T\1)\bigr)-G_P(R).
\end{equation}
Furthermore, if $(P,\pi)$ is identifiable and the columns of $R$ corresponding to positive
coordinates of $\pi$ are not identically zero,  then the right side is strictly positive unless $SR$ is a diagonal matrix
for some permutation matrix $S$.
\end{lem}

\begin{proof}
From the fact that $|(1-u)\log (1-u)+u|\le u^2$, for $0\le u\le 1$, it can be verified that,
$\bigl|\rho_n^{-1}\tau(\rho_n u) - \bigl(u\log \rho_n +\tau_0(u)\bigr)\bigr|\le \rho_n\ra 0$, uniformly in $0\le u\le 1$. It follows that, uniformly in $R$,
$$\frac{1}{\rho_n} H_{\rho_n P}(R) =\log\rho_n\sum_{a,b}(RPR^T)_{ab}+
\sum_{a,b}(R\1)_a(R\1)_b\tau_0\Bigl(\frac{(RPR^T)_{ab}}{(R\1)_a(R\1)_b}\Bigr)+O(\rho_n).$$
The first term on the right is equal to $\log \rho_n(R^T\1)^TP(R^T\1)$, and hence is the same for $R$ and $\Diag(R^T\1)$.
Thus this term cancels on taking the difference to form the left side of \eqref{EqApproximateHbyG},
and hence \eqref{EqApproximateHbyG} follows.

The right side of \eqref{EqApproximateHbyG} is nonnegative, because the left side is, by
Lemma~\ref{lem:maximality}. This fact can also be proved directly along the lines
of the proof of Lemma~\ref{lem:maximality}, as follows. Write
$$G_P(R)=\sum_{a,b}\sum_{a',b'}R_{aa'}R_{bb'}\Bigl[P_{a'b'}\log
\frac{(RPR^T)_{ab}}{(R\1)_a(R\1)_b}- \frac{(RPR^T)_{ab}}{(R\1)_a(R\1)_b}\Bigr].$$
By the information inequality for two Poisson distributions 
the term in square brackets becomes bigger if  $(RPR^T)_{ab}/(R\1)_a(R\1)_b$ is replaced by
$P_{a'b'}$. It then becomes $\tau_0(P_{a'b'})$ and the double sum on $(a,b)$ can be executed
to see that the resulting bound is $G_P\bigl(\Diag(R^T\1)\bigr)$.
Furthermore, the inequality is strictly unless \eqref{EqNoLossR} holds, with $\bar R=R$. 
Since also $G_P(P_\sigma R)=G_P(R)$, for every permutation matrix $P_\sigma$, the final assertion of the
lemma is proved by copying the proof of Lemma~\ref{lem:maximality}.
\end{proof}

\subsection{Strong consistency}
We need slightly adapted versions of the function $H_P$, given by,
with $\d_{ab}$ equal to 1 or 0 if $a=b$ or not,
\begin{align}\label{eq:HPn}
H_{P,n}(R)&=\frac{1}{2}\sum_{a, b}(R\1)_a\bigl((R\1)_b-\d_{ab}/n\bigr)\, \tau\Bigl(\frac{(RPR^T)_{ab}-\d_{ab}  \sum_{k} P_{kk} R_{ka}/n
}{(R\1)_a\bigl((R\1)_b-\d_{ab} /n\bigr)}\Bigr).
\end{align}
For given functions $t_{ab}: [0,1]\to\mathbb{R}$, let $X(e)$ be the $K\times K$ matrix with entries
\begin{equation}
\label{EqDefX}
X_{ab}(e)=t_{ab}\Bigl(\frac{\wt O_{ab}(e)}{n^2}\Bigr)-t_{ab}\Bigl(\frac{\E(\wt O_{ab}(e)\mid Z)}{n^2}\Bigr).
\end{equation}

\noindent\textbf{Proof of Theorem \ref{thm:strongconsistency} [strong consistency]}
\begin{proof}
\ref{thm:strongi}. By Theorem~\ref{thm:weakconsistency}, $\widehat e$ is weakly consistent, and hence
with probability tending to one it belongs to the set
of classifications $e$ such that the fractions $f(e)$ are close to $\pi$, and the matrices $R(e,Z)$ are close to
$\Diag(\pi)$ after the appropriate permutation of the labels (that is, of rows of $R(e,Z)$). 
Therefore, it is no loss of generality to assume that $\widehat e$ is restricted
to this set. By Lemmas~\ref{lem:O} and~\ref{lem:Oconditionalexpectation}, the matrices $\wt O(e)/n^2$ are then close to 
$R(e,Z)PR(e,Z)^T\ra \Diag(\pi)P\Diag(\pi)$, 
and hence are bounded away from zero and one if $P$ has this property.

If $\widehat e$ and $Z$ differ at $m$ nodes, then $\widehat e$ belongs to the set of $e$ with 
$\|R(Z,Z)-R(e,Z)\|_1=m(2/n)$, by Lemma~\ref{lem:Rdiscrepancy}. In that case 
$Q_B(e)\ge Q_B(Z)$, for some $e$ in this set, and hence by Lemma~\ref{lem:bayesstirling} 
$Q_{ML}(e)-Q_{ML}(Z)+Q_P(e)-Q_P(Z)\ge -\eta_n$, for some $\eta_n$ of order $(\log n)/n^2$. 
It follows that:
\begin{align}\nonumber
&\bigl[Q_{ML}(e)-H_{P,n}\bigl(R(e,Z)\bigr)\bigr]-\bigl[Q_{ML}(Z)-H_{P,n}\bigl(R(Z,Z)\bigr)\bigr]\\
&\qquad\qquad \ge H_{P,n}\bigl(R(Z,Z)\bigr)-H_{P,n}\bigl(R(e,Z)\bigr) -|Q_P(e)-Q_P(Z)|- \eta_n.
\label{EqLBoundModulus}
\end{align}
The first term on the right is bounded below by a multiple of $m/n$,
by Lemmas ~\ref{LemmaStrongMaximality} and \ref{lem:Rdiscrepancy}. Because $(x+\alpha)\log x-(y+\alpha)\log y=\int_x^y(\log s+(s+\alpha)/s)\,ds$
is bounded in absolute value by a multiple of $|x-y|\log (x\vee y)$, if $\alpha\ge0$ and $x,y>0$,
the second term $-|Q_P(e)-Q_P(Z)|$ is bounded below by a multiple of $m(\log n)/n^2$, for some positive constant $C_2$,
which is of smaller order than $m/n$.
We conclude that the left side of \eqref{EqLBoundModulus} is bounded below by
$C_1m/n$. The left side is $\sum_{a,b}\bigl(X_{ab}(e)-X_{ab}(Z)\bigr)$, for $X$ defined
in \eqref{EqDefX} and $t$ the function with coordinates
$t_{ab}(o)=f_a(e) \bigl(f_b(e)-\d_{ab}/n\bigr)\tau\bigl(o/f_a(e) \bigl(f_b(e)-\d_{ab}/n\bigr) \bigr)$.
Because we restrict $e$ to classifications such that $O_{ab}(e)/n_{ab}(e)$ and $f_a(e)f_b(e)$ are bounded away
from zero and one, only the values of the function $\tau$ on an open interval strictly within
$(0,1)$ matter. On any such interval $\tau$ has uniformly bounded derivatives, and hence 
the  bound of  Lemma~\ref{LemmaXmod} is valid. Thus we find that 
\begin{align*}
\Pr\bigl(\#(i: \widehat e_i\not= Z_i)=m\bigr)
& \le\Pr\Bigl(\sup_{e: \#(i: e_i\not= Z_i)\le m}\bigl\| X(e)-X(Z)\bigr\|_\infty\ge\frac{C_1m}n\Bigr)\\
&\lesssim K^m\binom{n}{m} e^{-cm^2/(m\|P\|_\infty/n+m/n)}\\
&\le e^{m\log(Kne/m)-c_1mn}.
\end{align*}
The sum of the right side over $m=1,\ldots,n$ tends to zero.

\ref{thm:strongii}. We follow the proof for \ref{thm:strongi}, but in \eqref{EqLBoundModulus} 
use that $H_{P,n}\bigl(R(Z,Z)\bigr)-H_{P,n}\bigl(R(e,Z)\bigr)\ge \rho_n C \|R(Z,Z)-R(e,Z)\|_1\ge \rho_n C2m/n$,
by Lemma~\ref{LemmaStrongMaximalitySmallRho}. Since $\rho_n \gg (\log n)/n$ by assumption,
we have that the contribution $m(\log n)/n^2$ of $Q_P(e)-Q_P(Z)$  is still negligible and hence
$\rho_n C2m/n$ is a lower bound for the left side of \eqref{EqLBoundModulus}.
As a bound on the left side of the preceding display, we then obtain
$$\sum_{m=1}^n K^m\binom{n}{m}e^{-c_2 \rho_n^2m^2/(m\rho_n/n+\rho_n m/n)}
\le \sum_{m=1}^ne^{m\log (Kne/m)-c_3\rho_n mn}.$$
This sum tends to zero provided that $n\rho_n \gg \log n$.
\end{proof}

\begin{lem}
\label{LemmaStrongMaximality}
If $P$ is fixed and symmetric and every pair of rows of $P$ is different and $0<P<1$ and $\pi>0$, then, for sufficiently
small $\delta>0$,
\begin{align}
\label{EqBCLinearAssumption}
\liminf_{n\ra\infty}\inf_{0<\|R-\Diag(\pi)\|<\delta}\frac{H_{P,n}\bigl(\Diag(R^T\1)\bigr)-H_{P,n}(R)}{\|\Diag(R^T\1)-R\|}>0.
\end{align}
\end{lem}

\begin{proof}
We can reparametrize the $K\times K$ matrices $R$ by the pairs $(R^T\1,R-\Diag(R^T\1))$,
consisting of the $K$ vector $f=R^T\1$ and the $K\times K$ matrix $R-\Diag(R^T\1)$. The latter matrix
is characterized by having nonnegative off-diagonal elements and zero column sums,
and can be represented in the basis consisting of all $K\times K$ matrices 
$\Delta_{bb'}$, for $b\not=b'$, defined by: $(\Delta_{bb'})_{b'b'}=-1$,
$(\Delta_{bb'})_{bb'}=1$ and $(\Delta_{bb'})_{aa'}=0$, for all other entries $(a,a')$, i.e.\
the $b'$th column of $\Delta_{bb'}$ has a $1$ in the $b$th coordinate and a $-1$ on the $b'$th coordinate
and all its other columns are zero. Given any matrix $R\ge 0$ the matrix $R-\Diag(R^T\1)$ can be decomposed as
$$R-\Diag(R^T\1)=\sum_{b\not= b'}\l_{bb'}\Delta_{bb'},$$
for $\l_{bb'}=R_{bb'}\ge 0$.  Since every $\Delta_{bb'}$ has exactly one nonzero off-diagonal element,
which is equal to 1, and in a different location for each $b\not=b$, the sum of the off-diagonal elements of the matrix on
the right side is $\sum_{b,b'}\l_{bb'}$. Because the sum of all its elements is zero,
it follows that its sum of absolute elements is given by $\|R-\Diag(R^T\1)\|_1=2\sum_{b\not= b'}\l_{bb'}$. 

Thus we obtain a further reparametrization $R\leftrightarrow (f,\l)$, in which
$R=\Diag(f)+\sum_{b\not= b'}\l_{bb'}\Delta_{bb'}$. For  given $P$, $f$ and $n$, define the function
$$G(\l)=H_{P,n}\Bigl(\Diag(f)+\sum_{b\not= b'}\l_{bb'}\Delta_{bb'}\Bigr).$$
Then we would like to show that there exists $C$ such that 
\begin{align*}
\frac{H_{P,n}(\Diag(R^T\1))-H_{P,n}(R)}{\|R-\Diag(R^T\1)\|_1}
=\frac{G(0)-G(\l)}{2\sum_{b\not=b'}\l_{bb'}}
&\ge C>0,
\end{align*}
for every $f$ in a neighbourhood of $\pi$, $\l$ in a neighbourhood of $0$ intersected with $\{\l: \l\ge 0\}$, and
every sufficiently large $n$. 
The numerator in the quotient is $f(0)-f(1)$ for the function $f(s)=G(s\l)$. Writing this difference in the form
$-f'(0)-\int_0^1\bigl(f'(s)-f'(0)\bigr)\,ds$ gives that the numerator is equal to 
\begin{align}
\label{EqExpansionG}
-\nabla G(0)^T\l -\int_0^1 \bigl(\nabla G(s\l)-\nabla G(0))\bigr)^T\,ds\, \l.
\end{align}
It suffices to show that the first term is bounded below by a multiple of $\|\l\|_1$ and
that the second is negligible relative to the first, as $n\ra\infty$,
uniformly in $f$ in a neighbourhood of $\pi$ 
and $\l$ in a neighbourhood of 0 intersected with $\{\l: \l\ge 0\}$. Thus it is sufficient to show first that
for every coordinate $\l_{bb'}$ of $\l$ minus the partial derivative of $G$ at $\l=0$ with respect to $\l_{bb'}$ is bounded away
from 0, as $n\ra\infty$ uniformly in $f$, and second that every partial derivative is
equicontinuous at $\l=0$ uniformly in $f$ and large $n$.

We have 
\begin{equation}
\label{EqDefG}
G(\l)=\frac{1}{2}\sum_{a, a'}f_a(\l)\bigl(f_{a'}(\l)-\d_{aa'}/n\bigr)\, 
\tau\Bigl(\frac{\bigl(R(\l)PR(\l)^T\bigr)_{aa'}-\d_{aa'} e_a(\l)/n}{f_a(\l)\bigl(f_{a'}(\l)-\d_{aa'}/n\bigr)}\Bigr),
\end{equation}
for 
\begin{align*}
f(\l)&=f+\sum_{bb'}\l_{bb'}(\Delta_{bb'}\1),\\
R(\l)&=\Diag(f)+\sum_{b\not= b'}\l_{bb'}\Delta_{bb'},\\
e_a(\l)&=\sum_kP_{kk}R_{ak}(\l)=P_{aa}f_a+\sum_{b\not= b'}P_{b'b'}\l_{bb'}(\d_{ab}-\d_{ab'}) .
\end{align*}
By a lengthy calculation, given in Lemma \ref{lem:lengthycalculation}, 
\begin{equation}
\label{EqDerivativeG}
\frac{\partial}{\partial\l_{bb'}}G(\l)_{|\l=0}=-\sum_a f_{a} K(P_{ab'} \|P_{ab}) +\frac{1}{2n}K(P_{b'b'}\| P_{bb}),
\end{equation}
for $K(p\|q)=p\log (p/q)+(1-p)\log \bigl((1-p)/(1-q)\bigr)$ the Kullback-Leibler divergence between
the Bernoulli distributions with success probabilities $p$ and $q$.
The numbers $f_a$ are bounded away from zero for $f$ sufficiently close to $\pi$, and hence
so is $\sum_a f_{a}  K(P_{ab'} \| P_{ab})$, unless the $b$th and $b'$th column of $P$ are identical.
The whole expression is bounded below by the minimum over $(b,b')$ of these numbers minus 
$(2n)^{-1}$ times the maximum of the numbers $K(P_{b'b'}\| P_{bb})$, and hence is positive and bounded
away from zero for sufficiently large $n$.

To verify the equicontinuity  of the partial derivatives we can compute these explicitly at 
$\l$ and take their limit as $n\ra\infty$.  We omit the details of this calculation.
However, we note that every term of $G(\l)$ is a fixed function of the quadratic forms in $\l$
\begin{align}
&\bigl(f_a+\sum_{bb'}\l_{bb'}(\Delta_{bb'}\1)_a\bigr)\bigl(f_{a'}+\sum_{bb'}\l_{bb'}(\Delta_{bb'}\1)_{a'}-\d_{aa'}/n\bigr),
\label{EqFunctionOne}\\
&\Bigl(\bigl(\Diag(f)+\sum_{b\not= b'}\l_{bb'}\Delta_{bb'}\bigr)P \bigl(\Diag(f)+\sum_{b\not= b'}\l_{bb'}\Delta_{bb'}^T\bigr)\Bigr)_{aa'}\nonumber\\
&\qquad\qquad\qquad\qquad\qquad\qquad-\frac {\d_{aa'}}{2n}\bigl(P_{aa}f_a+\sum_{b\not= b'}P_{b'b'}\l_{bb'}(\d_{ab}-\d_{ab'}) \bigr).
\label{EqFunctionTwo}
\end{align}
These forms are obviously smooth in $\l$, and their dependence and that of their derivatives
on $n$ is seen to vanish as $n\ra\infty$. For $f$ and $\l$ restricted to neighbourhoods of $\pi$ and
0, the values of the quadratic forms are restricted to a domain in which the transformation mapping them
into $G(\l)$ is continuously differentiable. Thus the desired equicontinuity follows by the chain rule.
\end{proof}

\begin{lem}
\label{lem:lengthycalculation}
The partial derivatives of the function $G$ at $0$ defined by \eqref{EqDefG} are given by 
\eqref{EqDerivativeG}.
\end{lem}

\begin{proof}
For given differentiable functions $u$ and $v$ the map $\eps\mapsto u(\eps)\tau\bigl(v(\eps)/u(\eps)\bigr)$ has derivative
$v'\log \bigl(v/(u-v)\bigr)-u'\log \bigl(u/(u-v)\bigr)$.
We apply this for every given pair $(a,a')$ to the functions $u$ and $v$ obtained by taking $\l_{bb'}$ in 
\eqref{EqFunctionOne} and \eqref{EqFunctionTwo}
equal to $\eps$ and all other coordinates of $\l$ equal to zero. Then 
\begin{align*}
u(0)&=f_a(f_{a'}-\d_{aa'}/n),\\
v(0)&=f_a(f_{a'}-\d_{aa'}/n)P_{aa'},\\
u'(0)&=(\Delta_{bb'}\1)_a(f_{a'}-\d_{aa'}/n)+f_a(\Delta_{bb'}\1)_{a'}\\
v'(0)&=(\Delta_{bb'}P)_{aa'}f_{a'}+f_a(\Delta_{bb'}P)_{a'a}-(\d_{aa'}/n) P_{b'b'}(\d_{ab}-\d_{ab'}).
\end{align*}
It follows that ${v(0)}/{(u(0)-v(0))}={P_{aa'}}/{(1-P_{aa'})}$,
and ${u(0)}/{(u(0)-v(0))}={1}/{(1-P_{aa'})}$. Hence in view of \eqref{eq:HPn}
the partial derivative in \eqref{EqDerivativeG} is equal to 
$$\sum_{a\not=a'}\Bigl[v'(0) \log \frac {P_{aa'}}{1-P_{aa'}}-u'(0)\log \frac {1}{1-P_{aa'}}\Bigr].$$
We combine this with the equalities
$$(\Delta_{bb'}\1)_a=
\begin{cases}
0& \text{ if } a\notin\{b,b'\},\\
-1& \text{ if } a=b',\\
1& \text{ if } a=b,
\end{cases}
\quad
(\Delta_{bb'}P)_{aa'}=
\begin{cases}
0& \text{ if } a\notin\{b,b'\},\\
-P_{b'a'}& \text{ if } a=b',\\
P_{b'a'}& \text{ if } a=b.
\end{cases}$$
\end{proof}

\begin{lem}
\label{LemmaStrongMaximalitySmallRho}
If $S$ is fixed and symmetric, every pair of rows of $S$ is different and $S>0$ and $\pi>0$ coordinatewise, then
there exists $C>0$ such that, for sufficiently small $\delta>0$ and any $\rho_n\downarrow 0$,
\begin{align*}
\liminf_{n\ra\infty}\inf_{0<\|R-\Diag(\pi)\|<\delta}\frac{H_{\rho_n S,n}\bigl(\Diag(R^T\1)\bigr)-H_{\rho_n S,n}(R)}{\rho_n \|\Diag(R^T\1)-R\|}\ge C.
\end{align*}
\end{lem}

\begin{proof}
In the notation of the proof of Lemma~\ref{LemmaStrongMaximality} we must now show that
$G(0)-G(\l)\ge C \rho_n\|\l\|_1$, as $n\ra\infty$, uniformly in $f$ in a neighbourhood of $\pi$, and $\l$ in a 
positive neighbourhood of $0$. As in that proof we write $G(0)-G(\l)$ in the
form \eqref{EqExpansionG} and see that it suffices that the partial derivatives
of $G$ at 0 divided by $\rho_n$ tend to  negative limits, and that $\bigl\|\nabla G(\l)-\nabla G(0)\bigr\|/\rho_n$
becomes uniformly small as $\l$ is close enough to zero.

The partial derivative at 0 with respect to $\l_{bb'}$ is given in \eqref{EqDerivativeG}, where we must
replace $P$ by $\rho_n S$.
Since the scaled Kullback-Leibler divergence
$\rho_n^{-1}K(\rho_n s \| \rho_n t)$ of two Bernoulli laws converges to the Kullback-Leibler divergence
$K_0(s\|t)=s\log (s/t)+ t-s$  between two Poisson laws of means $s$ and $t$, as $\rho_n\ra 0$,
it follows that for $\rho_n\ra0$, uniformly in $f$,
$$\frac1 {\rho_n}\frac{\partial}{\partial\l_{bb'}} G(\l)_{|\l=0}\ra 
-\sum_a f_a K_0(S_{ab'}\|S_{ab}).$$
The right side is strictly negative by the assumption that every pair of rows of $S$  differ in 
at least one coordinate.

If $P=\rho_n S$, then the function $\l\mapsto v(\l)$ given in \eqref{EqFunctionTwo} takes the form $v=\rho_n v_S$,
for $v_S$ defined in the same way but with $S$ replacing $P$. The function $u$ given in 
\eqref{EqFunctionOne} does not depend on $P$ or $S$. Using again that
the derivative of the map $\eps\mapsto u(\eps)\tau\bigl(v(\eps)/u(\eps)\bigr)$ is given by 
$v'\log \bigl(v/(u-v)\bigr)-u'\log \bigl(u/(u-v)\bigr)$, we see that the partial derivative 
with respect to $\l_{bb'}$ of the $(a,a')$ term in the sum defining $G$ takes the form
$$\rho_n v_S' \log \frac{\rho_n v_S}{u-\rho v_S}- u'\log \frac {u}{u-\rho_n v_S}
=\rho_n v_S'\log \rho_n -\rho_nv_S'\log (v_S/u)-(\rho_n v_S'-u')\log (1-\rho_n v_S/u).$$
Here $u$ and $V_S$ are as in \eqref{EqFunctionOne} and \eqref{EqFunctionTwo} (with $P$ replaced by $S$),
and depend on $(a,a')$.
From the fact that the column sums of the matrices $R(\l)$ do not depend on $\l$,
we have that
$$\sum_{a,a'}\bigl[(R(\l)S R(\l)^T)_{aa'}-\frac{\d_{aa'}}{n}\sum_k P_{kk}R(\l)_{ak}\bigr]
=R(\l)^T\1SR(\l)^T\1-\sum_k P_{kk} \sum_a R(\l)_{ak}$$
is constant in $\l$. This shows that $\sum_{a,a'}v_S'=0$ and hence the contribution
of the term $\rho_n v_S'\log \rho_n$ to the partial derivatives of $G$ vanishes. 
The  term $-(\rho_n v_S'-u')\log (1-\rho_n v_S/u)$ can be expanded
as $(\rho_n v_S'-u')\rho_n v_S/u$ up to  $O(\rho_n^2)$, uniformly in $f$ and $\l$.
Since these are equicontinuous functions of $\l$, it follows that
$\rho_n^{-1} \bigl(\nabla G(\l)-\nabla G(0)\bigr)$ becomes arbitrarily small if $\l$
varies in a sufficiently small neighbourhood of $0$. 
\end{proof}

\begin{lem}
\label{LemmaXmod}
There exists a  constant $c>0$ such that for $X(e)$ as in \eqref{EqDefX}, for every 
twice differentiable functions $t_{a,b}: [0,1]\to \mathbb{R}$ with $\|t_{a,b}'\|_\infty\vee \|t_{a,b}''\|_\infty\le1$,
and every $x>0$, 
\begin{align*}
&\Pr\Bigl(\max_{e: \# (e_i\not=Z_i)\le  m}\bigl\|X(e)-X(Z)\bigr\|_\infty> x\Bigr)
&\le 6 \binom{n}{m}K^{m+2} e^{-\frac{cx^2n^2}{m\|P\|_\infty/n+x}}.
\end{align*}
\end{lem}

\begin{proof}
Given $Z$ there are at most $\binom{n}{m}$ groups of $m$ candidate nodes that can
be assigned to have $e_i\not= Z_i$, and the label of each node
can be chosen  in at most $K-1$ ways. Thus conditioning
the probability on $Z$, we can use the union bound to pull out the maximum over $e$,
giving a sum of fewer than $\binom{n}{m}K^m$  terms. Next we pull out the norm
giving another factor $K^2$. It suffices to combine this with a tail bound for 
a single variable $X_{a,b}(e)-X_{a,b}(Z)$. Write $t$ for $t_{a,b}$.

Assume for simplicity of notation that $e_i=Z_i$, for $i>m$, and decompose
\begin{align*}
\frac1{n^2}O_{ab}(e)
&=\frac 1{n^2}\Bigl[\sum_{i\le m \text{ or }j\le m}A_{ij}1_{e_i=a,e_j=b}+\sum_{i> m \text{ and }j > m}A_{ij}1_{e_i=a,e_j=b}\Bigr]\\
&=: S_1+S_2.
\end{align*}
Let $O_{ab}(Z)/n^2=: S_1'+S_2$, with the same variable $S_2$, be the corresponding
decomposition if $e$ is changed to $Z$, and then decompose, where the expectation signs
$\E$ denote conditional expectations given $Z$,
\begin{align*}
X_{ab}&(e)-X_{ab}(Z)\\
&=\bigl(t(S_1+S_2)-t(\E S_1+\E S_2)\bigr)-\bigl(t(S_1'+S_2)-t(\E S_1'+\E S_2)\bigr)\\
&=t(S_1+S_2)-t(\E S_1+S_2)\\
&\quad+\bigl(t(\E S_1+S_2)-t(\E S_1+\E S_2)\bigr)-\bigl(t(\E S_1'+S_2)-t(\E S_1'+\E S_2)\bigr)\\
&\quad+t(\E S_1'+S_2)-t(S_1'+ S_2)
\end{align*}
The first and third terms on the far right can be bounded above in absolute value by $\|t'\|_\infty$ times
the increment. To estimate the second term we  write it as
$$(S_2-\E S_2)(\E S_1-\E S_1')\int_0^1\!\int_0^1 t''\bigl(u S_2+(1-u)\E S_2+v\E S_1+(1-v)\E
S_1'\bigr)\,du\,dv.$$
Since the first and second derivatives of $t$ are uniformly bounded by 1, it follows that
$$\bigl|X_{ab}(e)-X_{ab}(Z)\bigr|\le |S_1-\E S_1|+ |S_2-\E S_2|\,|\E S_1-\E S_1'|+ |S_1'-\E S_1'|.$$
The variable $S_1-\E S_1$ is a sum of fewer than  $2mn$ independent variables, each with
conditional mean zero, bounded above by $1/n^2$ and of variance bounded above
by $\|P\|_\infty/n^4$. Therefore Bernstein's inequality gives that
$$\mathbb{P}\bigl( |S_1-\E S_1|>x\bigr)\le e^{-\tfrac{1}{2} x^2/(2mn\|P\|_\infty/n^4+x/(3n^2))}.$$
This is as the exponential factor in the bound given by the lemma, for appropriate $c$.
The variable $S_1'-\E S_1'$ can be bounded similarly. Furthermore
$|\E S_1-\E S_1'|\le 4mn/n^2=4m/n$, and $S_2-\E S_2$ is the sum of fewer than $n^2$ variables
as before, so that 
$$\mathbb{P}\bigl( |S_2-\E S_2|\,|\E S_1-\E S_1'|>x\bigr)\le e^{-\tfrac{1}{2} (xn/(4m))^2/(n^2\|P\|_\infty/n^4+xn/(12mn^2))}.$$
The exponent has a similar form as before, except for an additional factor $n/m\ge 1$.
\end{proof}

\bibliographystyle{imsart-nameyear}
\bibliography{bayesiansbm}

\end{document}